\documentclass[a4paper,10pt]{amsart}
\usepackage[utf8]{inputenc}
\usepackage[T1]{fontenc}
\usepackage{amsxtra}
\usepackage{amsopn}
\usepackage{amsmath,amsthm,amssymb}
\usepackage{amscd}
\usepackage{amsfonts}
\usepackage{latexsym}
\usepackage{verbatim}
\usepackage{MnSymbol}
\usepackage{xcolor}
\usepackage{tikz-cd}
\usepackage{hyperref}
\usepackage[shortlabels]{enumitem}
\usepackage[normalem]{ulem}

\newcommand{\lla}{\llangle}
\newcommand{\rra}{\rrangle}
\newcommand{\lv}{\lVert}
\newcommand{\rv}{\rVert}
\theoremstyle{plain}
\newtheorem{theorem}{Theorem}[section]
\newtheorem*{theorem*}{Theorem}

\newtheorem{lemma}[theorem]{Lemma}
\newtheorem{proposition}[theorem]{Proposition}
\newtheorem{corollary}[theorem]{Corollary}
\newtheorem{remark}[theorem]{Remark}

\newtheorem{problem}[theorem]{Problem}

\newtheorem*{mt*}{Main Theorem}
\newcommand\C{{\mathbb C}}

\renewcommand\phi{{\varphi}}

\newcommand\N{{\mathbb N}}

\newcommand\R{{\mathbb R}}

\newcommand\U{{\mathcal U}}

\renewcommand\H{{\mathcal H}}

\newcommand{\del}{\partial}
\newcommand{\delbar}{{\overline{\del}}}

\newcommand{\cinf}{\mathcal{C}^\infty}
\DeclareMathOperator{\supp}{supp}
\let\inf\undefined
\DeclareMathOperator*{\inf}{inf\vphantom{p}}
\let\sup\undefined
\DeclareMathOperator*{\sup}{sup\vphantom{p}}

\DeclareMathOperator{\img}{Im}

\DeclareMathOperator{\spec}{spec}

\DeclareMathOperator{\im}{{im\,}}
\DeclareMathOperator{\coker}{{coker\,}}
\DeclareMathOperator{\vol}{{Vol}}

\let\phi\varphi
\let\c\overline

\title{$L^2$ Fr\"olicher inequalities}

\author{Francesco Bei}
\address{Dipartimento di Matematica \lq\lq Guido Castelnuovo\rq\rq,
Sapienza Universit\`{a} di Roma, Piazzale Aldo Moro 5,
I-00185 Roma, Italy}
\email{bei@mat.uniroma1.it}
\author{Riccardo Piovani}
\address{Dipartimento di Matematica e Fisica \lq\lq Niccolò Tartaglia\rq\rq,
Universit\`{a} Cattolica del Sacro Cuore, Via della Garzetta 48,
25133 Brescia, Italy}
\email{riccardo.piovani@unicatt.it}

\thanks{\newline The first-named author was partially supported by 2024 Sapienza research grant “New
research trends in Mathematics at Castelnuovo”. Both authors are also partially supported by INdAM - GNSAGA Project,
codice CUP E53C24001950001.
}

\begin{document}

\begin{abstract}
We prove a Fr\"olicher inequality between $L^2$ Betti and $L^2$ Hodge numbers on normal coverings of compact complex manifolds. This is achieved by building an injection using suitable spectral projectors associated to the self-adjoint operators $(D_h)^2:=(\delbar+\delbar^*+h\del+h\del^*)^2$ for $h\in[0,1]$. With similar techniques, we show that the positivity of the spectrum of the Dolbeault Laplacian implies the positivity of the spectrum of the Hodge Laplacian; moreover, if equality holds in the $L^2$ Fr\"olicher inequality, then we can replace \lq\lq positivity of the spectrum\rq\rq{} with \lq\lq spectral gap at 0\rq\rq{} in the previous statement.
As a by-product, in the case of compact complex manifolds, we find a new proof of the classical Fr\"olicher inequality which does not rely at all on spectral sequences and build an explicit injection from de Rham to Dolbeault cohomology.
\end{abstract}

\maketitle

\section{Introduction}

A fundamental relation between topological and complex invariants of a compact complex manifold $M$ is provided by the Fr\"olicher inequality \cite{Fro}
\begin{equation}\label{equation frolicher inequality}
b^k(M)\le\sum_{p+q=k}h^{p,q}_\delbar(M)
\end{equation}
for all $k\in\N$, where $b^k(M)$ and $h^{p,q}_\delbar(M)$ denote respectively Betti and Hodge numbers.
The only proof known to the authors is the original one, which relies on spectral sequences. In this paper, we generalise \eqref{equation frolicher inequality} to normal coverings of compact complex manifolds \cite{Ati,G,Lu}. The proof is completely different from the original: it is analytical and yields interesting new consequences in both the setting of coverings and that of compact manifolds.

We begin introducing the setting and providing some motivation for the study of \eqref{equation frolicher inequality}.
Let us recall that on a complex manifold the exterior derivative on the space of smooth $(p,q)$-forms $A^{p,q}$ decomposes as $d=\del+\delbar$, so that $d^2=0$ is equivalent to the three relations $\del^2=\delbar^2=\del\delbar+\delbar\del=0$, which define Dolbeault, $\del$, Aeppli and Bott-Chern cohomology spaces
\begin{align*}
&H^{p,q}_\delbar(M):=\frac{A^{p,q}\cap\ker\delbar}{\im\delbar},&& H^{p,q}_\del(M):=\frac{A^{p,q}\cap\ker\del}{\im\del},\\
&H^{p,q}_A(M):=\frac{A^{p,q}\cap\ker\del\delbar}{\im\del+\im\delbar},&&H^{p,q}_{BC}(M):=\frac{A^{p,q}\cap\ker\del\cap\ker\delbar}{\im\del\delbar}.
\end{align*}
When the manifold is compact, the dimensions of these cohomology spaces are finite and will be denoted respectively by $h^{p,q}_\delbar(M)$, $h^{p,q}_\del(M)$, $h^{p,q}_A(M)$ and $h^{p,q}_{BC}(M)$. 
On a complex manifold, the identity map induces the following diagram of maps.
\begin{equation}\label{equation diagram maps cohomology}
\begin{tikzcd}
{}&{H^{\bullet,\bullet}_{BC}(M)}\arrow[dl]\arrow[d]\arrow[dr] &{}\\
{{H^{\bullet,\bullet}_\del(M)}}\arrow[dr]&{H^\bullet_{dR}(M;\C)}\arrow[d] &{H^{\bullet,\bullet}_\delbar(M)}\arrow[dl]\\
{}&{H^{\bullet,\bullet}_A(M)}&{}
\end{tikzcd}
\end{equation}
When all these maps are isomorphisms, the complex manifold is said to satisfy the $\del\delbar$-Lemma, and by \cite[Lemma 5.15]{DGMS} this is equivalent to 
\[
\im\del\delbar=\ker\del\cap\ker\delbar\cap\im d.
\]
It is well known that compact K\"ahler manifolds satisfy the $\del\delbar$-Lemma.

Angella and Tomassini \cite{AT} proved that on a compact complex manifold $M$
\begin{equation}\label{equation angella tomassini inequality hodge abc}
h^{p,q}_\delbar(M)+h^{p,q}_\del(M)\le h^{p,q}_{A}(M)+h^{p,q}_{BC}(M)
\end{equation}
for all $p,q\in\N$, so that the combination of inequalities \eqref{equation frolicher inequality}, \eqref{equation angella tomassini inequality hodge abc} produces
\begin{equation}\label{equation angella tomassini inequality betti abc}
2b^k(M)\le\sum_{p+q=k}h^{p,q}_{A}(M)+h^{p,q}_{BC}(M)
\end{equation}
for all $k\in\N$.
They also proved that equality in \eqref{equation angella tomassini inequality betti abc} holds for all $k\in\N$ if and only if the $\del\delbar$-Lemma holds. This provides a useful numerical condition to check the validity of the $\del\delbar$-Lemma.

If a complex manifold $M$ is endowed with a Hermitian metric $g$, there is a well defined $L^2$ inner product on the space of smooth $(p,q)$-forms $A^{p,q}$
\[
\lla \alpha,\beta\rra_{L^2}:=\int_M g(\alpha,\beta)\vol=\int_M\alpha\wedge*\overline{\beta}.
\]
In the above formula, $g$ denotes the Hermitian metric extended on the space of $(p,q)$-forms, $\vol$ is the standard volume form $\frac{\omega^n}{n!}$, where $\omega$ is the fundamental 2-form of the metric and $n$ is the complex dimension of the complex manifold, and finally $*$ indicates the $\C$-linear Hodge $*$ operator. We define the Hilbert spaces of $L^2$ complexified $k$- or $(p,q)$-forms as the completion of the space of smooth compactly supported complexified $k$- or $(p,q)$-forms with respect to the $L^2$ inner product. These spaces will be denoted by $L^2\Lambda^k_\C$ and $L^2\Lambda^{p,q}$.
The $L^2$ formal adjoints of the operators $d$, $\del$ and $\delbar$ are then respectively $d^*:=-*d*$, $\del^*:=-*\delbar*$ and $\delbar^*:=-*\del*$. Let us define, on the space of complexified forms $A^\bullet_\C:=\oplus_{k\ge0}A^k_\C$, the elliptic and formally self-adjoint differential operators
\begin{align*}
D_1:=d+d^*, && D_0:=\delbar+\delbar^*,
\end{align*}
so that 
\begin{align*}
D_1^2:=dd^*+d^*d, && D_0^2:=\delbar\delbar^*+\delbar^*\delbar
\end{align*}
are respectively the Hodge and Dolbeault Laplacians.
If the complex manifold is compact, Hodge theory implies that
\begin{align*}
\ker D_1^2\simeq H^\bullet_{dR}(M;\C), && \ker D_0^2\simeq H^{\bullet,\bullet}_\delbar(M).
\end{align*}
These two kernels are called, respectively, spaces of de Rham and Dolbeault harmonic forms.

If the Hermitian manifold $(\widetilde M,\tilde g)$ is not compact, one may consider the spaces of $L^2$ de Rham and Dolbeault harmonic forms
\begin{align*}
\ker D_1^2\cap L^2\Lambda^\bullet_\C, && \ker D_0^2\cap L^2\Lambda^{\bullet,\bullet},
\end{align*}
but their vector space dimensions are in general not finite. However, when $\pi:\widetilde M\to M$ is a normal $\Gamma$-covering (\textit{i.e.}, the group of deck transformations $\Gamma$ acts transitively on all the fibres and $\widetilde M$ is connected) (\textit{e.g.}, the universal covering) of a compact complex manifold $M\simeq {\widetilde M}/{\Gamma}$, and $\tilde g=\pi^*g$ is the pull-back of any Hermitian metric $g$ on $M$, by \cite{Ati} one can \lq\lq renormalise\rq\rq{} their dimensions to obtain finite invariants
\begin{align*}
b^k_{\Gamma}(M):=\dim_\Gamma(\ker D_1^2\cap L^2\Lambda^k_\C)\in\R_{\ge0}, && h^{p,q}_{\delbar,\Gamma}(M):=\dim_{\Gamma}(\ker D_0^2\cap L^2\Lambda^{p,q})\in\R_{\ge0},
\end{align*}
which only depend on $M$ and on the covering (no metric dependence). The invariant $h^{p,q}_{\del,\Gamma}$ is defined similarly. This renormalised dimension is called $\Gamma$-dimension or Von Neumann dimension, whose definition will be recalled in Section \ref{section l2 frolicher}. These invariants are called $L^2$ Betti and $L^2$ Hodge numbers and, if $\pi:M\to M$ is the trivial covering, coincide with the usual Betti and Hodge numbers. It has been proved in \cite{Dod} that $b^k_{\Gamma}(M)$ is a homotopy invariant of the pair $(M,\Gamma)$. For a complex manifold $M$, it follows by \cite[Corollary 11.4]{K} that $h^{p,0}_{\delbar,\Gamma}(M)$ is a bimeromorphic invariant. These invariants have been largely studied in both the Riemannian and the complex Hermitian setting and turned out to be key ingredients for many interesting applications, see, \textit{e.g.}, \cite{Lu,G,K}.

Holt and the second author \cite{HP} have shown that, in the same setting as before, the mininal and maximal closed extensions of the operator $\del\delbar$ coincide, and as a consequence there are well defined $L^2$ Aeppli and Bott-Chern numbers $h^{p,q}_{A,\Gamma}(M)$, $h^{p,q}_{BC,\Gamma}(M)$. They also generalised \eqref{equation angella tomassini inequality hodge abc} to the $L^2$ case, proving
\begin{equation}\label{equation l2 inequality hodge abc}
h^{p,q}_{\delbar,\Gamma}(M)+h^{p,q}_{\del,\Gamma}(M)\le h^{p,q}_{A,\Gamma}(M)+h^{p,q}_{BC,\Gamma}(M)
\end{equation}
for all $p,q\in\N$. The proof uses Varouchas exact sequences as in \cite{AT}, for which Von Neumann dimension has good properties. Motivated by the above inequality and by the results of \cite{AT}, in \cite[Section 10.2]{HP} it is asked if a $L^2$ version of the Fr\"olicher inequality holds.

In the present paper we answer this question by establishing the following result.
\begin{theorem}[{See Theorem \ref{theorem l2 frolicher}}]\label{theorem intro l2 frolicher inequality}
    Let $\pi:\widetilde M\to M$ be a normal $\Gamma$-covering of a compact complex manifold. Then for all $k\in\N$
    \begin{equation}\label{equation intro l2 frolicher inequality}
    b^k_\Gamma(M)\le\sum_{p+q=k}h^{p,q}_{\delbar,\Gamma}(M).
    \end{equation}
\end{theorem}
The classical Fr\"olicher inequality \eqref{equation frolicher inequality} is proved via the so called Fr\"olicher spectral sequence, which definition heavily relies on the structure of double complex of $(A^{\bullet,\bullet},\del,\delbar)$.
This approach seems difficult to generalise to the $L^2$ setting for the following two reasons: first, we are interested in the $L^2$ reduced cohomology (since it is naturally isomorphic to the space of $L^2$ harmonic forms) instead of the usual $L^2$ cohomology; second, we have to take care of the domains of our operators (\textit{e.g.}, a $k$-form in the domain of $d$ is not necessarily in the domains of either $\del$ or $\delbar$).
We were therefore led to build a proof of the $L^2$ Fr\"olicher inequality which did not rely at all on  spectral sequences.
Being these $L^2$ invariants defined through the spaces of $L^2$ harmonic forms, we first looked for a new proof of the original Fr\"olicher inequality on compact complex manifolds which used harmonic forms.

The idea is the following.
As in \cite{Po} and the references thereof, we consider a smooth family of elliptic and formally self-adjoint differential operators on $A^\bullet_\C$ joining $D_0$ to $D_1$, namely
\[
D_h:=\delbar+\delbar^*+h\del+h\del^*
\]
for $h\in[0,1]$. It follows from well known results that, on a compact manifold, the dimension of the kernel of $D_h$ is upper-semi-continuous in $h$ (see, \textit{e.g.}, \cite[Chapter 4, Theorem 4.3]{KM} with complex and Hermitian structures independent of $t=h$, or more generally \cite[Chapter XI.4, Theorem 4.1]{GGK}), meaning that
\begin{equation}\label{equation intro upper semi continuity} 
\dim(\ker D_0^2)\ge\dim(\ker D_h^2)
\end{equation}
for $h$ sufficiently close to 0. We are then left to show that
\begin{equation}\label{equation intro isomorphism kernel popovici}  
\ker D_h^2\simeq\ker D_1^2
\end{equation}
for all $h\in(0,1)$, which follows as in \cite[Corollary 2.9]{Po}. One defines $\theta_h:A^{p,q}\to A^{p,q}$ as $\theta_h(\alpha):=h^p\alpha$ and extends to $A^k_\C$ by linearity. Then $d_h:=\delbar+h\del=\theta_hd\theta_h^{-1}$ satisfies $d_h^2=0$ and defines a cohomology. Therefore
\[
\ker D_1^2\simeq\frac{\ker d}{\im d}\overset{
}\simeq\frac{\ker d_h}{\im d^h}\simeq\ker D_h^2,
\]
where the first and last isomorphisms follow by standard Hodge theory and the second one is induced by $\theta_h$.
This completes the new analytical proof of the original Fr\"olicher inequality \eqref{equation frolicher inequality} on a compact complex manifold.

The situation on a normal $\Gamma$ covering $\pi:\widetilde M\to M$ of a compact complex manifold $M$ is more complicated. It is not difficult to replace \eqref{equation intro isomorphism kernel popovici} with a similar $\Gamma$-equivariant isomorphism between the spaces of $L^2$ harmonic forms, but since the spectrum of $D_h^2$, as a self-adjoint operator between the Hilbert spaces of $L^2$ forms, is no more discrete on the non-compact manifold $\widetilde M$, we lose completely the upper-semi-continuity \eqref{equation intro upper semi continuity}. We first tried to replace it by building an injection
\begin{equation}\label{equation intro injection kernel l2} 
\ker(D_h^2)\cap L^2\Lambda^\bullet_\C\to\ker(D^2_0)\cap L^2\Lambda^\bullet_\C,
\end{equation}
but we found it possible only when the image of $D_0^2$ on $L^2\Lambda^\bullet_\C$ is closed (see Corollary \ref{corollary injective projection}). Finally, by considering the spectral family associated to $D_h^2$, we obtained a more suitable injection using some spectral projectors (see Theorem \ref{theorem injective projection} and Lemma \ref{lemma covering convergence norm resolvent sense}), allowing us to get an inequality for the spectral density functions (see Proposition \ref{proposition l2 frolicher spectral density function}), which, passing to the limit, proves Theorem \ref{theorem intro l2 frolicher inequality}. 

We pause for a moment to focus on some further consequences of the techniques developed in the proof of Theorem \ref{theorem intro l2 frolicher inequality}. First, the functional analytic Theorem \ref{theorem injective projection} is very general and will likely find other applications (cf. Subsection \ref{subsection deformations}). Second, the inequality of spectral density functions of Proposition \ref{proposition l2 frolicher spectral density function} has the following important corollary.
\begin{corollary}[{See Theorems \ref{theorem positive spectrum},  \ref{theorem spectral gap}}]\label{corollary introduction spectral}
Let $\pi:\widetilde M\to M$ be a normal $\Gamma$-covering of a compact complex manifold. If the Dolbeault Laplacian has entirely positive spectrum, then the same holds for the Hodge Laplacian. Moreover, if equality holds in \eqref{equation intro l2 frolicher inequality} and the Dolbeault Laplacian has a spectral gap at  $0$, then the same spectral gap holds for the Hodge Laplacian.
\end{corollary}

Back to the case of the trivial covering $\pi: M\to M$, since the image of $D_0^2$ on $L^2\Lambda^\bullet_\C$ is closed, \eqref{equation intro injection kernel l2} holds. Therefore, once we fix a Hermitian metric $g$, we can construct an explicit injection from De Rham cohomology to Dolbeault cohomology. Again, let us remark that this injection is built without any need of spectral sequences.
\begin{corollary}[{See Corollary \ref{corollary compact injection cohomology}}]
Let $(M,g)$ be a compact Hermitian manifold. Then there is an explicit injection
    \[
    H^k_{dR}(M)\to \bigoplus_{p+q=k}  H^{p,q}_{\delbar}(M).
    \]
\end{corollary}
We also notice that, on a compact complex manifold $M^n$, the other Fr\"olicher relation \cite{Fro}
\begin{equation}\label{equation intro frolicher relation euler}
\sum_{k=1}^{2n}(-1)^kb^k(M)=\sum_{p,q=1}^n(-1)^{p+q}h^{p,q}_\delbar(M),
\end{equation}
originally proved using spectral sequences, can alternatively be proved by perturbation theorems for Fredholm operators \cite[Chapter XI.4, Theorem 4.1]{GGK} applied to the family $D_h$. A generalised $L^2$ Fr\"olicher relation \eqref{equation intro frolicher relation euler} for $L^2$ Betti and $L^2$ Hodge numbers was proved in \cite{Bei-Von}.

These notes are organised as follows. In Section \ref{section preliminaries} we briefly recall some preliminaries about $L^2$ Hodge theory for elliptic complexes of first order differential operators, manifolds of bounded geometry and Sobolev spaces. In Section \ref{section projectors} we prove a general result of functional analysis:  convergence in the norm resolvent sense of a family of non-negative self-adjoint operators allows to build a key injection using spectral projectors. In Section \ref{section l2 frolicher} we consider the case of a normal covering of a compact complex manifold, where we prove the convergence in the norm resolvent sense of $(D_h)^2\to(D_0)^2$ and thus obtain the $L^2$ Fr\"olicher inequality. We also generalise to the $L^2$ setting the inequality $b^2(M)\le 2h^{0,2}_\delbar(M)+h^{1,1}_{BC}(M)$ due to Kodaira and Spencer, proving
\[
b^2_\Gamma(M)\le 2h^{0,2}_{\delbar,\Gamma}(M)+h^{1,1}_{BC,\Gamma}(M),
\]
in Theorem \ref{theorem kodaira spencer l2}.
Finally, in Section \ref{section open problems} we list a couple of open problems.

\medskip\medskip
\noindent{\em Acknowledgments.} 
The second author would like to thank Daniele Angella, Anna Fino, Jonas Stelzig and Adriano Tomassini for useful conversations.

\section{Preliminaries on $L^2$ Hodge theory and bounded geometry}\label{section preliminaries}

\subsection{$L^2$ Hodge theory for complete metrics}
Given a Riemannian manifold $(M^m,g)$ and $n+1$ Hermitian vector bundles $(E_j,h_j)$, $j=0,\dots,n$, we consider a \emph{complex} of $\C$-linear differential operators $D_j$
\[
\cinf(M,E_0)\overset{D_0}{\longrightarrow}\cinf(M,E_1)\overset{D_1}{\longrightarrow}\dots\overset{D_{n-2}}{\longrightarrow}\cinf(M,E_{n-1})\overset{D_{n-1}}{\longrightarrow}\cinf(M,E_{n}),
\]
namely $D_{j+1}\circ D_{j}=0$ for all $j$. Here and later in these notes $\cinf(M,E)$ denotes smooth sections from $M$ to $E$. The complex is said to be \emph{elliptic} if it induces an exact sequence of principal symbols, namely if for all $j$
\[
\im\sigma_{D_j}(x,\xi)=\ker\sigma_{D_{j+1}}(x,\xi),
\]
for all $x\in M$ and $0\ne\xi\in T_x^*M$, where $\sigma_{D_j}(x,\xi)$ is seen as a linear map $(E_j)_x\to(E_{j+1})_x$. We recall that a differential operator $P$ is \emph{elliptic} if its principal symbol $\sigma_P(x,\xi)$ is an isomorphism for all $x\in M$ and $0\ne\xi\in T_x^*M$.
We will restrict to the case where all $D_j$ are differential operators of order 1. Then the complex is elliptic if and only if all the Laplacians
\[
\Delta_j:=D_j^*D_j+D_{j-1}D_{j-1}^*:\cinf(M,E_j)\to \cinf(M,E_j)
\]
are elliptic, where $D_j^*$ denotes the $L^2$ formal adjoint of $D_j$, if and only if the direct sum operator
\[
D:=D_0\oplus (D_0^*+D_1)\oplus\dots\oplus(D_{n-2}^*+D_{n-1})\oplus D_{n-1}^*:\cinf(M,E_\bullet)\to \cinf(M,E_\bullet)
\]
is elliptic. This last fact follows from
\[
D^2=\Delta:=\oplus_{j=0}^n\Delta_j.
\]
Notice that $D$ and $\Delta_j$ are formally self-adjoint.

Denote by $L^2(M,E_j)$ the space of $L^2$ sections of $E_j$, namely the completion of smooth compactly supported sections  $\cinf_c(M,E_{j})$ with respect to the $L^2$ norm $(\int_Mh_j(\cdot,\cdot)\gamma)^\frac12$, where $\gamma$ denotes the standard Riemannian density. Let us assume that the metric $g$ is complete, and that the norm of the principal symbol of the first order differential operator $D$ is uniformly bounded as follows: there is $C>0$ such that
\[
|\sigma_{D}(x,\xi)|\le C(1+|\xi|)
\]
for all $x\in M$. Then by \cite[Theorem 2.2]{Che}, the restrictions to $\cinf_c(M,E_{\bullet})$ of $D$ and of every power $D^k$, for $k\in\N$, are \emph{essentially self-adjoint}, namely they have a unique self-adjoint extension as unbounded operators $L^2(M,E_\bullet)\to L^2(M,E_\bullet)$. One can similarly prove that, under the same assumptions, the minimal and the maximal closed extensions of the operators $D_j$ and $D_j^*$ coincide \cite{Gaf}. In the following, we will abuse notation by indicating this unique choice of closed $L^2$ extensions by the same symbols $D$, $D^k$, $\Delta$, $D_j$, $D_j^*$, $\Delta_j$. We refer to \cite{HP} and the references therein for a review of the theory of unbounded operators between Hilbert spaces and minimal and maximal closed extensions of differential operators.

Under the same assumptions, the above elliptic complex gives rise to a \emph{Hilbert complex} \cite{BL} of closed and densely defined unbounded operators
\[
L^2(M,E_0)\overset{D_0}{\longrightarrow}L^2(M,E_1)\overset{D_1}{\longrightarrow}\dots\overset{D_{n-2}}{\longrightarrow}L^2(M,E_{n-1})\overset{D_{n-1}}{\longrightarrow}L^2(M,E_{n}),
\]
producing $L^2$-orthogonal Hodge decompositions \cite[Lemma 2.1]{BL}
\begin{align*}
L^2(M,E_j)&=\ker (\Delta_j)\oplus\overline{\im (D_{j-1})}\oplus\overline{\im (D_j^*)},\\
\ker (D_j)&=\ker (\Delta_j)\oplus\overline{\im (D_{j-1})},
\end{align*}
where
\[
\ker (D_j)\cap\ker (D_{j-1}^*)=\ker (\Delta_j)\subseteq \cinf(M,E_j).
\]
Here the equality is \cite[Proposition 7]{AV}, while the inclusion is elliptic regularity. Defining the reduced cohomology space as
\[
L^2\bar H (M,E_j):=\frac{\ker (D_j)}{\overline{\im(D_{j-1})}},
\]
the $L^2$ Hodge decomposition above implies that the identity map induces an isomorphism
\[
\ker (\Delta_j)\to L^2\bar H (M,E_j).
\]

\subsection{Bounded geometry and Sobolev spaces}

In this subsection we follow \cite{Sh}. Given a Riemannian manifold $(M^m,g)$, we denote by $\exp_x:T_xM\to M$ the usual exponential geodesic map, which is a diffeomorphism of a ball $B_x(0,r)\subseteq T_xM$ of radius $r>0$ and centre 0 on a neighbourhood $\U_{x,r}$ of $x$ in $M$. Denoting by $r_x$ the supremum of the possible radii of such balls, we can define the \emph{injectivity radius} of $M$ as $r_{inj}=\inf_{x\in M}r_x$. If $r_{inj}>0$, then taking $r\in(0,r_{inj})$ we see that $\exp_x:B_x(0,r)\to\U_{x,r}$ is a diffeomorphism for all $x\in M$. Euclidean coordinates in $T_xM$ 
define coordinates on $\U_{x,r}$ by means of $\exp_x$ which are called \emph{canonical}.

A complete Riemannian manifold $(M^m,g)$ is called a manifold of \emph{bounded geometry} if
\begin{enumerate}
    \item $r_{inj}>0$;
    \item $|\nabla^k R|\le C_k$ for all $k\in\N$; namely every covariant derivative of the Riemannian curvature tensor is bounded.
\end{enumerate}
Property (2) can be replaced with the following equivalent condition:
\begin{itemize}
\item[(2')] given $r\in(0,r_{inj})$ and $\U_{x,r}$, $\U_{x',r}$ two domains of canonical coordinates $y:\U_{x,r}\to\R^m$, $y':\U_{x',r}\to\R^m$ such that $\U_{x,r}\cap\U_{x',r}\ne\emptyset$, then the vector function $y'\circ y^{-1}:y(\U_{x,r}\cap\U_{x',r})\to\R^m$ satisfies
\[
|\del^\alpha_y(y'\circ y^{-1})|\le C_{\alpha,r}
\]
uniformly for every multi-index $\alpha$.
\end{itemize}
Examples of manifolds of bounded geometry are Lie groups with invariant metrics and coverings of compact manifolds with pull-back metrics.

The idea is that we can use canonical coordinates to formulate a uniform notion of $\cinf$-boundedness. For example, a smooth function $f:M\to\C$ is called $\cinf$-bounded if $|\del^\alpha_yf(y)|\le C_\alpha$ for every multi-index $\alpha$ and for any choice of canonical coordinates. We refer to \cite{Sh} for the more general definition of $\cinf$-bounded exterior form and tensor field, as well as of complex vector bundle of bounded geometry, of $\cinf$-bounded differential operators between bundles of bounded geometry and of uniformly elliptic $\cinf$-bounded differential operators. 


On a manifold of bounded geometry, using canonical coordinates it is possible to build a uniform partition of unity.

\begin{lemma}[{\cite[Lemma 1.1]{Sh}}]\label{lemma shubin covering}
Given a manifold $M$ of bounded geometry, for every $\epsilon>0$ there exists a countable covering of $M$ by balls of radius $\epsilon$, $M =\cup_iB(x_i,\epsilon)$, such that the maximal number of the balls $B(x_i,2\epsilon)$ with non-empty intersection is finite.
\end{lemma}

\begin{lemma}[{\cite[Lemma 1.2]{Sh}}]\label{lemma shubin partition of unity}
Given a manifold $M$ of bounded geometry, let us fix $r\in(0,r_{inj})$. For every $\epsilon\in(0,r/2)$ there exists a partition of unity $1=\sum_{i=1}^{+\infty}\phi_i$ on $M$ such that
\begin{enumerate}
    \item $\phi_i\in\cinf_c(M,\R)$, $\phi_i\ge0$ and $\supp\phi_i\subseteq B(x_i,2\epsilon)$ for the $x_i$ of Lemma \ref{lemma shubin covering};
    \item $|\del_y^\alpha\phi_i|\le C_\alpha$ in canonical coordinates $y$, uniformly for every multi-index $\alpha$.
\end{enumerate}
\end{lemma}

This is useful in the definition of uniform Sobolev-Hilbert spaces $W^{s,2}(M)$, for $s\in\N$. First introduce the Sobolev norm $\lv\cdot\rv_{s,2}$ on $\cinf_c(M)$ by the formula
\[
\lv u\rv_{s,2}^2:=\sum_{i=1}^{+\infty}\lv\phi_i u\rv^2_{s,2;B(x_i,2\epsilon)},
\]
where $\lv\cdot\rv_{s,2;B(x_i,2\epsilon)}$ means the usual Sobolev norm of order $s$ in canonical coordinates on $B(x_i,2\epsilon)$, \textit{i.e.},
\[
\lv v\rv^2_{s,2;B(x_i,2\epsilon)}=\sum_{|\alpha|\le s}\int_{B(x_i,2\epsilon)}|\del^\alpha_y v(y)|^2dy.
\]
Then define $W^{s,2}(M)$ as the completion of smooth compactly supported functions $\cinf_c(M)$ with respect to the norm $\lv \cdot\rv_{s,2}$. Similarly one defines $W^{s,2}(M,E)$ for any bundle $E$ of bounded geometry. We refer to \cite{Sh} for other equivalent Sobolev norms. Notice that $W^{0,2}(M,E)=L^2(M,E)$.

The following is the fundamental elliptic inequality on manifolds of bounded geometry, which is directly obtained by the same inequality holding on the above cover of balls $B(x_i,2\epsilon)$.

\begin{lemma}[{\cite[Lemma 1.3]{Sh}}]\label{lemma shubin elliptic inequality}
Given a manifold $M$ of bounded geometry, let $A:\cinf(M,E)\to\cinf(M,F)$ be a $\cinf$-bounded uniformly elliptic differential operator of order $m$ with $E,F$ vector bundles of bounded geometry. Then for $s,t\in\N$ there exists $C>0$ such that for all $u\in\cinf_c(M,E)$
\begin{equation}\label{equation elliptic inequality}
\lv u\rv_{s,2}\le C \left( \lv Au\rv_{s-m,2}+\lv u\rv_{t,2}\right)
\end{equation}
\end{lemma}

\begin{proposition}[{\cite[Proposition 1.1]{Sh}}]\label{proposition shubin common domain}
In the same assumptions of Lemma \ref{lemma shubin elliptic inequality}, denoting by $A_{min}$ and $A_{max}$ the minimal and maximal closed extensions of $A$ as an unbounded operator $L^2(M,E)\to L^2(M,F)$, then their two domains coincide with $W^{m,2}(M,E)$; in particular $A_{min}=A_{max}$. In this case we will abuse notation and simply write $A$ instead of $A_{min}$ or $A_{max}$. Moreover, by \eqref{equation elliptic inequality}, the Sobolev norm $\lv \cdot\rv_{m,2}$ is equivalent to the graph norm of $A$.
\end{proposition}

\begin{corollary}[{\cite[Corollary 1.1]{Sh}}]\label{corollary shubin essentially self adjoint}
In the same assumptions of Lemma \ref{lemma shubin elliptic inequality}, if $E=F$ and $A$ is formally self-adjoint, then $A$ is essentially self-adjoint.
\end{corollary}

In Section \ref{section l2 frolicher}, we will work on normal coverings of compact complex manifolds. These manifolds have bounded geometry since the covering map is a local isometry. Moreover, the bundles of exterior forms are bundles of bounded geometry and all the differential operators that we will use will be $\cinf$-bounded, since these objects are the pull-back, or lift, of the same objects on the underlying compact manifolds. The condition of uniformly ellipticity is proved similarly.

\section{Injectivity of projectors}\label{section projectors}

For the notion of spectral family, or spectral resolution, of a self-adjoint operator we refer to \cite[Appendix A.5]{Gri}. To provide some intuition, note that when the spectrum $\spec P\subseteq\R$ of a self-adjoint operator $P:H\to H$ is discrete, its spectral family $E_\lambda:H\to H$ is the orthogonal projection onto the direct sum of $P$-eigenspaces $H_\sigma$ for $\sigma\le\lambda$. We recall the notion of convergence of self-adjoint operators in the norm resolvent sense following \cite[Section VIII.7]{RS1}.

Let $D$ be a closed operator on a Hilbert space $H$. A complex number $\lambda\in\C$ is in the \emph{resolvent set} of $D$ if $D-\lambda I$ is a bijection of the domain of $D$ onto $H$ with a bounded inverse. If $\lambda$ is in the resolvent set, the inverse map $R_\lambda(D):=(D-\lambda I)^{-1}$ is called the \emph{resolvent} of $D$ at $\lambda$. If $\lambda\in\C$ is not in the resolvent set of $D$, then $\lambda$ is said to be in the \emph{spectrum} $\spec D$ of $D$.

Let $D_h$ be self-adjoint operators on a Hilbert space $H$ for $h\in[0,1]$. Then $D_h$ is said to converge to $D_0$ in the \emph{norm resolvent sense} for $h\to0$ if $R_\lambda(D_h)\to R_\lambda(D_0)$ in norm for all $\lambda\in\C$ with $\img \lambda\ne0$.

We recall the following criterium for convergence in the norm resolvent sense which will be used in the next section.
\begin{lemma}[{\cite[Theorem VIII.25(b)]{RS1}}]\label{lemma reed simon}
Let $\{D_h\}_{h\in[0,1]}$ be self-adjoint operators on a Hilbert space $H$ with common domain $D$. Endow $D$ with the graph norm of $D_0$. If
\[
\sup_{\lv\phi\rv_H+\lv D_0\phi\rv_H=1}\lv (D_h-D_0)\phi\rv_H\to0
\]
for $h\to0$, then $D_h\to D_0$ in the norm resolvent sense.
\end{lemma}

We now present the main functional analytic result of this paper which will play a crucial role in the proof of Theorem \ref{theorem l2 frolicher}.

\begin{theorem}\label{theorem injective projection}
Let us consider a family of non-negative self-adjoint operators $\{D_h\}_{h\in[0,1]}$ on a Hilbert space $H$, each with spectral family $\{E_{h,\lambda}\}_{\lambda\in\R}$. If $D_h\to D_0$ in the norm resolvent sense for $h\to0$, then for every $\tau>0$ there exists an $h_\tau>0$ such that for all $0\le\sigma<\tau$ and $0<h<h_\tau$ the projection
\[
E_{0,\tau|\im (E_{h,\sigma})}:\im (E_{h,\sigma})\to\im (E_{0,\tau})
\]
is injective.
\end{theorem}

\begin{proof}
Since $D_h$ is non-negative, then $\spec D_h\subseteq[0,+\infty)$.
Throughout the proof we denote with $R_{h}:=(D_h+I)^{-1}$ the resolvent $R_{-1}(D_h)$ of $D_h$ at $-1$. By contradiction, let us assume that the above statement does not hold true. Then there exist real numbers $0<\sigma<\tau$ and a sequence $\{h_n\}_{n\in \mathbb{N}}$, $h_n\rightarrow 0$ as $n\rightarrow +\infty$, such that 
\[
E_{0,\tau|\im (E_{h_n,\sigma})}:\im (E_{h_n,\sigma})\to\im (E_{0,\tau})
\]
is not injective for all $n\in\N$.
This means that for each $n\in\N$ we can find an element $u_{n}$ such that $\|u_{n}\|_{H}=1$ and 
\[
u_{n}\in \im (E_{h_n,\sigma})\cap (\im (E_{0,\tau}))^{\bot}.
\]
Note now that as consequence of the spectral theorem for unbounded self-adjoint operators \cite[Appendix A.5.4]{Gri} we get the inequalities
\[
\frac{1}{1+\tau}\geq \lv R_{0}u_{n}\rv_{H}
\]
and 
\[
\lv R_{h_n}u_{n}\rv_H\ge \frac{1}{1+\sigma}.
\]
For the sake of clarity we prove the first inequality above. The second one follows by arguing in the same manner. 
We refer to \cite[Appendix A.5]{Gri} for the notation. In particular, the measure $\lv E_{0,\lambda}u_{n}\rv_{H}^2$ on $\R$ is defined on each interval $[a,b)$ as $\lv E_{0,b} u_{n}\rv_{H}^2-\lv E_{0,a} u_{n}\rv_{H}^2$.
We have
$$
\begin{aligned}
\lv R_{0}u_{n}\rv_{H}^2
&=
\int_{0}^{+\infty}\left(\frac{1}{1+\lambda}\right)^2d\lv E_{0,\lambda} u_{n}\rv_{H}^2\\
&=\int_{\tau}^{+\infty}\left(\frac{1}{1+\lambda}\right)^2d\lv E_{0,\lambda} u_{n}\rv_{H}^2\\
&\leq \left(\frac{1}{1+\tau}\right)^2\int_{\tau}^{+\infty}d\lv E_{0,\lambda}u_{n}\rv_{H}^2\\
&= \left(\frac{1}{1+\tau}\right)^2 \lim_{b\rightarrow +\infty}\int_{\tau}^bd\lv E_{0,\lambda}u_{n}\rv_{H}^2\\
&=\left(\frac{1}{1+\tau}\right)^2 \lim_{b\rightarrow +\infty}\left(\lv E_{0,b}u_{n}\rv_{H}^2-\lv E_{0,\tau} u_{n}\rv_{H}^2\right)\\
&=\left(\frac{1}{1+\tau}\right)^2\lv u_{n}\rv_{H}^2.
\end{aligned}
$$
All the steps are justified, in order, by: \cite[eq. A.50]{Gri}; since for $\lambda<\lambda'$ we have $\im (E_{0,\lambda})\subseteq \im (E_{0,\lambda'})$ \cite[Definition A.2]{Gri} and since $u_n\in(\im (E_{0,\tau}))^{\bot}$, then $E_{0,\lambda} u_{n}=0$ for $\lambda\in[0,\tau]$; elementary inequality; definition of $\int_\tau^{+\infty}$; definition of the measure; $E_{0,\lambda}$ tends to the identity in the strong operator topology \cite[Definition A.2]{Gri} and $E_{0,\tau} u_{n}=0$ as above.

We thus obtain, 
using the reverse triangle inequality,
$$
\begin{aligned}
\frac{1}{1+\tau}&\geq \|R_{0}u_{_n}-R_{h_n}u_{n}+R_{h_n}u_{n}\|_{H}\\
& \geq \left|\|R_{0}u_{n}-R_{h_n}u_{n}\|_{H}-\|R_{h_n}u_{n}\|_{H}\right|\\
& \ge\|R_{h_n}u_{n}\|_{H}-\|R_{0}u_{n}-R_{h_n}u_{n}\|_{H}\\
& \ge\frac{1}{1+\sigma}-\|R_{0}u_{n}-R_{h_n}u_{n}\|_{H}.\\
\end{aligned}
$$
Since $D_{h_n}\to D_0$ in the norm resolvent sense for $n\to+\infty$ and $-1$ is not in the spectrum of $D_0$, by \cite[Theorem VIII.23]{RS1} we deduce
\[
\|R_{0}u_{n}-R_{h_n}u_{n}\|_{H}\to0
\]
as $n\to+\infty$. 
Summarising, we have 
$$
\frac{1}{1+\tau}-\frac{1}{1+\sigma}\ge\lim_{n\rightarrow +\infty} \|R_{0}u_{n}-R_{h_n}u_{n}\|_{H}=0
$$
which is a contradiction since $0\le\sigma<\tau$. We can thus conclude that the statement of this theorem holds true.
\end{proof}

\begin{remark}\label{remark more general assumptions}
The same conclusion of Theorem \ref{theorem injective projection} holds in the following more general cases:
\begin{itemize}
    \item if we ask only for the non-negativity of $D_0$ (or of $D_0+kI$, for any positive constant $k>0$), instead of the non-negativity of the whole family $\{D_h\}$; this is true by applying \cite[Theorem VIII.23]{RS1};
    \item if we directly ask for the convergence of the resolvents $(D_h+kI)^{-1}\to (D_0+kI)^{-1}$ in operator norm as $h\to0$ for some positive constant $k>0$, instead of applying \cite[Theorem VIII.23]{RS1} to deduce it; this assumption replaces the non-negativity of the operators and the convergence in the norm resolvent sense.
\end{itemize}
\end{remark}

Given a non-negative self-adjoint operator $D$ on a Hilbert space $H$, we recall that the following two facts are equivalent:
\begin{enumerate}
    \item the image of $D$ is closed;
    \item $D$ has a \emph{spectral gap} at 0, \textit{i.e.}, $\inf(\spec(D)\setminus\{0\})=C>0$.
\end{enumerate}

\begin{corollary}\label{corollary closed image}
    In the same assumptions of Theorem \ref{theorem injective projection}, if $\im D_0$ is closed, then there exists an $h'>0$ such that for all $0<h<h'$ the projection
    \[
    E_{0,0|\ker (D_{h})}:\ker(D_{h})\to\ker(D_{0})
    \]
    is injective.
\end{corollary}
\begin{proof}
    Given $\inf(\spec(D_0)\setminus\{0\})=C>0$, then for all $0<\tau<C$
    \[
    \im(E_{0,\tau})=\ker D_0.
    \]
    The claim now follows by taking $\sigma=0$ in Theorem \ref{theorem injective projection}.
\end{proof}

\section{Normal coverings of compact complex manifolds}\label{section l2 frolicher}

Let $(M,g)$ be a compact Hermitian manifold along with a normal covering $\pi:\widetilde{M}\rightarrow M$ as defined in the introduction. Denote by $\Gamma$ the corresponding group of deck transformations.  In what follows we may call the covering a \emph{normal $\Gamma$-covering}. Denote by $\tilde{g}=\pi^*g$ the pull-back metric on the covering. Note that a pull-back metric is \emph{$\Gamma$-invariant}, \textit{i.e.}, $\gamma^*g=g$ for all $\gamma\in\Gamma$, and any $\Gamma$-invariant metric is a pull-back metric.

Given any Hermitian vector bundle $(E,h)$ on $M$, indicate by $(\widetilde E,\tilde h)$ the Hermitian pull-back bundle with respect to $\pi$. Note that the group $\Gamma$ acts in a natural way on $\widetilde{E}$. Indeed, given any $\gamma\in \Gamma$, $x\in \widetilde{M}$ and  $(x,v)\in \widetilde{E}_x$ with $v\in E_{\pi(x)}$, the action of $\gamma$ on $(x,v)$ is given by  $\gamma_E(x,v):=(\gamma(x),v)$. This gives clearly a linear isomorphism $\gamma_E:\widetilde{E}_x\rightarrow \widetilde{E}_{\gamma(x)}$ which is also an isometry with respect to $\tilde{h}$. Thus the group $\Gamma$ acts also on $C^{\infty}(\tilde{M},\widetilde{E})$  by $\gamma\cdot s:=\gamma_E\circ s\circ \gamma^{-1}$.
Let us consider now the space $L^2(\widetilde M,\widetilde E)$ of square-integrable sections of $\widetilde E$ with respect to the pull-back metric $\tilde h$. Note that $L^2(\widetilde M,\widetilde E)$, as a topological vector space, does not depend on the choice of the pull-back metric $\tilde h$.
The group $\Gamma$ then acts on $L^2 (\widetilde M,\widetilde E)$ by isometries given by the action $\gamma\cdot\alpha$ for any $\gamma \in \Gamma$, $\alpha \in L^2 (\widetilde M,\widetilde E)$.  

A closed Hilbert subspace $V \subseteq L^2 (\widetilde M,\widetilde E)$ is called a \emph{Hilbert $\Gamma$-module} if it is preserved by the action of $\Gamma$, \textit{i.e.}, $\gamma\cdot V=V$ for any $\gamma \in \Gamma$. 
Maps between Hilbert $\Gamma$-modules are given by bounded \emph{$\Gamma$-equivariant} operators, namely bounded operators $P$ which commute with the action with respect to any $\gamma \in\Gamma$, \textit{i.e.}, $\gamma\cdot(P\alpha)=P(\gamma\cdot\alpha)$ for any $\alpha \in L^2 (\widetilde M,\widetilde E)$. 
A sequence of maps between Hilbert $\Gamma$-modules
\[
\dots\longrightarrow V_{i-1}\overset{d_{i-1}}{\longrightarrow} V_i\overset{d_{i}}{\longrightarrow} V_{i+1}\longrightarrow\dots
\]
is called \emph{weakly exact} if $\c{\im d_{i-1}}=\ker d_i$ for all $i$.

We will now define the $\Gamma$-dimension, or Von Neumann dimension, of a Hilbert $\Gamma$-module $V \subseteq L^2 (\widetilde M,\widetilde E)$. We refer to \cite{Ati}, \cite[Section 1.1]{Lu}, \cite[Section 9]{HP} for a more detailed treatment.

First take an orthonormal basis $\{\phi_i \}$ of $V$ and define the function on $\widetilde{M}$ with values in $\R_{\ge0}\cup\{+\infty\}$
$$
\tilde{f}(\tilde{x}) = \sum_i \tilde{h}_{\tilde x}(\phi_i (\tilde x),\phi_i (\tilde x)). 
$$
This definition is independent of the choice of orthonormal basis $\{\phi_i \}$, and moreover $\tilde f$ is $\Gamma$-invariant, so that it descends to a function $f$ on the quotient $M\simeq \Gamma\backslash\widetilde M$.
The $\Gamma$-dimension is then given by
$$
\dim_{\Gamma} V = \int_M f(x) d\mu
$$
and takes values in $\R_{\ge0}\cup\{+\infty\}$. If $\pi:\widetilde M\to M$ is a finite normal $\Gamma$-covering, then the $\Gamma$-dimension coincides with the usual dimension of a vector space divided by the cardinality of $\Gamma$ \cite[p. 45]{Ati}. In particular, if $\pi:M\to M$ is the trivial covering, then the Von Neumann dimension coincides with the usual dimension of a vector space.

We state a couple of properties of the $\Gamma$-dimension which will be used in the following. We refer to \cite[Theorem 1.12]{Lu} for the proof.

\begin{proposition}[Properties of the Von Neumann dimension]\label{proposition properties von neumann dimension}
For a Hilbert $\Gamma$-module $V$, it holds
\[
V=\{0\} \iff \dim_\Gamma(V)=0.
\]
Moreover, if $0\to U\to V\to W\to 0$ is a weak exact sequence of Hilbert $\Gamma$-modules, then
\[
\dim_\Gamma(V)=\dim_\Gamma(U)+\dim_\Gamma(W).
\]
\end{proposition}

\subsection{Proof of the $L^2$ Fr\"olicher inequalities}

In the same setting as before, denote by $L^2\Lambda^{k}\widetilde{M}$, $L^2\Lambda^{k}_\C\widetilde{M}$ and $L^2\Lambda^{p,q}\widetilde{M}$ the Hilbert spaces of $L^2$ $k$-forms, complexified $k$-forms and $(p,q)$-forms on $\widetilde M$ with respect to the pull-back Hermitian metric $\tilde g$.
Each bundle of forms on the covering $\widetilde{M}$, \textit{e.g.}, $\Lambda^{k}\widetilde{M}$, endowed with the natural extension to forms of the metric $\tilde g$, is isometrically isomorphic to the relative pull-back bundle, \textit{e.g.}, $\pi^*\Lambda^{k}{M}$, endowed with the natural pull-back metric.

For $h\in[0,1]$, we consider on the algebra of complexified forms $A^\bullet_\C=\cinf(\widetilde{M},\Lambda^{\bullet}_\C\widetilde{M})$ the operators
\[
D_h:=\delbar+\delbar^*+h\del+h\del^*,
\]
which are elliptic and formally self-adjoint. Thanks to the completeness of the pull-back metric, the restrictions to smooth and compactly supported forms of the operators $D_h$ on $(\widetilde{M},\tilde g)$ are essentially self-adjoint, and the same holds for their positive integer powers $D_h^m$ \cite[Theorem 2.2]{Che}. For simplicity, we indicate this unique self-adjoint extension on $L^2\Lambda^{\bullet}_\C\widetilde{M}$ of $D_h^m$ by the same symbol.

The operators $D_h^m$ 
are $\Gamma$-equivariant, therefore the kernels $\ker D_h^m\subseteq L^2\Lambda^{\bullet}_\C\widetilde{M}$ are Hilbert $\Gamma$-modules. We can then compute the Von Neumann dimensions 
\[
\dim_\Gamma (\ker D_h^m),
\]
which are finite since $D_h^m$ are elliptic and the manifold $M$ is compact by \cite[Propositions 2.4, 4.15]{Ati}. 
Furthermore, this dimension does not depend on the choice of the pull-back metric, therefore depending only on the compact manifold $M$ and on the covering, whose dependence is denoted by $\Gamma$.

Let us denote by $(D_h^m)^k$ the restriction of $D_h^m$ to $k$-forms. In particular,
$(D_1^2)^k$ is the Hodge Laplacian $dd^*+d^*d$ restricted to $k$-forms, and $(D_0^2)^k$ is the sum for $p+q=k$ of the Dolbeault Laplacians $\delbar\delbar^*+\delbar^*\delbar$ restricted to $(p,q)$-forms.
Note that
\[
\dim_\Gamma (\ker (D_1)^k)=b^k_{\Gamma}(M),\ \ \ \ \ \dim_\Gamma (\ker (D_0)^k)=\sum_{p+q=k}h^{p,q}_{\delbar,\Gamma}(M),
\]
where $b^k_{\Gamma}(M)$ and $h^{p,q}_{\delbar,\Gamma}(M)$ are the $L^2$-Betti and Hodge numbers as defined in the introduction.
For the other values of $h\in(0,1)$ we have the following result. Its proof is a generalisation of
\cite[Corollary 2.9]{Po}.
\begin{proposition}\label{proposition popovici 1}
Let $\pi:\widetilde{M} \rightarrow M$ be a normal $\Gamma$-covering of a compact complex manifold $M$. Let $g$ be an arbitrary Hermitian metric on $M$ and endow $\widetilde M$ with the pull-back metric $\tilde g$. For all $h\in(0,1)$ we have an isomorphism of Hilbert $\Gamma$-modules
\[
\ker (D_h)^k\simeq \ker (D_1)^k.
\]
Consequently
\[
\dim_\Gamma (\ker (D_h)^k)=b^k_{\Gamma}(M).
\]
\end{proposition}
\begin{proof}
The metric $\tilde g$ is complete, therefore the minimal and maximal closed extensions respectively of $d$ and 
\[
d_h:=\delbar+h\del
\]
coincide by Section \ref{section preliminaries}. We denote these extensions, which are closed and densely defined operators between the Hilbert spaces $L^2\Lambda^k_\C \widetilde M\to L^2\Lambda^{k+1}_\C \widetilde M$, by the same symbols $d$ and $d_h$. Note that $d_h^2=0$. Therefore, on $(\widetilde M, \tilde g)$ there are well defined reduced cohomology spaces
\begin{align*}
&L^2 \bar H^k_{dR,\Gamma}(M):=\frac{L^2\Lambda^{k}_\C\widetilde M\cap\ker d}{\c{\im d}}, &L^2 \bar H^k_{d_h,\Gamma}(M):=\frac{L^2\Lambda^{k}_\C\widetilde M\cap\ker d_h}{\c{\im d_h}}.
\end{align*}
We prove that these two spaces are isomorphic for $h\in(0,1)$.
For $h>0$ we define the bounded and bijective operator
\begin{align*}
&\theta_h:L^2\Lambda^{p,q}\widetilde M\to L^2\Lambda^{p,q}\widetilde M &\alpha\mapsto h^p\alpha.
\end{align*}
A direct verification shows that $\theta_h$ induces an isomorphism
\begin{equation}\label{equation isomorphism de rham dh}   
L^2 \bar H^k_{dR,\Gamma}(M)\simeq L^2 \bar H^k_{d_h,\Gamma}(M),
\end{equation}
which actually is an isomorphism of Hilbert $\Gamma$-modules.
In fact, $d_h\theta_h u=\theta_h du$, therefore $\theta_hu\in\ker d_h\iff u\in\ker d$, and $\theta_hu=d_hv$ if and only if $u=d(\theta_h^{-1}v)$, so $\theta_hu\in\im d_h\iff u\in\im d$ and the same equivalence passes to the closures.

By standard $L^2$ Hodge theory (see Section \ref{section preliminaries}) there are isomorphisms induced by the identity
\begin{align*}
&\ker (D_1)^k\simeq L^2 \bar H^k_{dR,\Gamma}(M), &\ker (D_h)^k\simeq L^2 \bar H^k_{d_h,\Gamma}(M)
\end{align*}
which actually are isomorphisms of Hilbert $\Gamma$-modules. This ends the proof.
\end{proof}

Denote by $\{E_{h,\lambda}^k\}_{\lambda\in\R}$ the spectral family of the non-negative self-adjoint operator $(D_h^2)^k$. Since $(D_h^2)^k$ is $\Gamma$-equivariant, each projection $E_{h,\lambda}^k$ commutes with the action of $\Gamma$. Therefore, since $\im(E_{h,\lambda}^k)$ is closed, we get that $\im(E_{h,\lambda}^k)$ is a $\Gamma$-module in $L^2\Lambda^k_\C\widetilde M$. The spectral density function $N_{h,\Gamma}^k(\lambda)$ for $\lambda\in\R$ can then be defined (\cite[eq (1.1)]{GS}, cf. \cite[Lemma 2.3]{Lu}) as
\[
N_{h,\Gamma}^k(\lambda):=\dim_\Gamma(\im (E_{h,\lambda}^k)).
\]
From the definition it follows that $N_{h,\Gamma}^k(0)=\dim_\Gamma (\ker (D_h^2)^k)=\dim_\Gamma (\ker (D_h)^k)$ and $N_{h,\Gamma}^k(\tau)\ge N_{h,\Gamma}^k(\sigma)$ for $\tau\ge\sigma$. Moreover \cite[p. 376]{GS} (cf. \cite[Theorem 1.12]{Lu})
\[
\lim_{\lambda\to0^+}N_{h,\Gamma}^k(\lambda)=\dim_\Gamma (\ker (D_h)^k).
\]
Notice that, if $\pi:M\to M$ is the trivial covering, then the spectrum of $(D_h^2)^k$ is discrete and $N_{h,\Gamma}^k(\lambda)$ boils down to the sum of the dimensions of all eigenspaces $H_\sigma$ of $(D_h^2)^k$ such that $\sigma\le\lambda$.

We now show an inequality between $N_{h,\Gamma}^k$ and $N_{0,\Gamma}^k$, by verifying that the functional analytic assumptions of Theorem \ref{theorem injective projection} are satisfied in our geometric setting.

\begin{lemma}\label{lemma covering convergence norm resolvent sense}
Let $\pi:\widetilde{M}\to M$ be a normal $\Gamma$-covering of a compact complex manifold. Let $g$ be an arbitrary Hermitian metric on $M$ and endow $\widetilde M$ with the pull-back metric $\tilde g$. Then $D_h\to D_0$ and $D_h^2\to D_0^2$ in the norm resolvent sense for $h\to0$. 
\end{lemma}
\begin{proof}
The first claim is proved if we manage to apply Lemma \ref{lemma reed simon} to the family of self-adjoint operators $D_h$, for $h\in[0,1]$. They have a common domain by Proposition \ref{proposition shubin common domain}. Moreover, for each $h\in[0,1]$ the graph norm of $D_h$ is equivalent to the graph norm of the operator $\del+\del^*$. The remaining assumption that we have to check is that, as $h\to0$,
\begin{equation}\label{equation 1 lemma convergence}
\sup_{\lv \phi\rv_{L^2}+\lv D_0\phi\rv_{L^2}=1} \rv(D_h-D_0)\phi\rv_{L^2}\to0
\end{equation}
with $\phi$ in the domain of $D_0$.
First note that for a smooth and compactly supported form $\phi$ we have
\[
\lv(D_h-D_0)\phi\rv_{L^2}=h\lv \del\phi+\del^*\phi\rv_{L^2}\le h(\lv\phi\rv_{L^2}+\lv \del\phi+\del^*\phi\rv_{L^2})\le  h C (\lv\phi\rv_{L^2}+\lv D_0\phi\rv_{L^2}),
\]
where $C>0$ is the constant given by the equivalence of the graph norms of $\del+\del^*$ and $D_0$. Since smooth and compactly supported forms are dense in the domain of $D_h$ for all $h\in[0,1]$ with respect to the corresponding graph norm, we can conclude that the above inequality extends to the whole domain of $D_0$. Now \eqref{equation 1 lemma convergence} follows immediately.

The second claim is proved precisely with the same strategy: the operators $D_h^2$ for $h\in[0,1]$ have a common domain by Proposition \ref{proposition shubin common domain}, and for each $h\in[0,1]$ the graph norm of $D_h^2$ is equivalent to the graph norm of $\del\del^*+\del^*\del$; furthermore
\begin{equation}\label{equation 2 lemma convergence}
\sup_{\lv \phi\rv_{L^2}+\lv D_0^2\phi\rv_{L^2}=1} \rv(D_h^2-D_0^2)\phi\rv_{L^2}\to0
\end{equation}
with $\phi$ in the domain of $D_0^2$.
More precisely, for a smooth and compactly supported form $\phi$
\begin{align*}
\lv(D_h^2-D_0^2)\phi\rv_{L^2}&=\lv h^2(\del\del^*+\del^*\del)\phi+h(\delbar\del^*+\del^*\delbar+\del\delbar^*+\delbar^*\del)\phi\rv_{L^2}\\
&\le h^2(\lv\phi\rv_{L^2}+\lv (\del\del^*+\del^*\del)\phi\rv_{L^2})+h\lv(\delbar\del^*+\del^*\delbar+\del\delbar^*+\delbar^*\del)\phi\rv_{L^2}\\
&\le h^2 K(\lv\phi\rv_{L^2}+\lv D_0^2\phi\rv_{L^2})+hK'\lv\phi\rv_{2,2}\\
&\le h(h K+K'K'')(\lv\phi\rv_{L^2}+\lv D_0^2\phi\rv_{L^2}),
\end{align*}
where $K>0$ is the constant given by the equivalence of the graph norms of $\del\del^*+\del^*\del$ and $D_0^2$, $K'>0$ is a constant given by the definition of the Sobolev norm in canonical coordinates, and $K''>0$ is the constant given by the equivalence of the Sobolev norm with the graph norm of $D_0^2$. Since smooth and compactly supported forms are dense in the domain of $D_h^2$ for all $h\in[0,1]$ with respect to the corresponding graph norm, we can conclude that the above inequality extends to the whole domain of $D_0^2$. Now \eqref{equation 2 lemma convergence} follows immediately.
\end{proof}

\begin{corollary}
    Let $\pi:\widetilde{M}\to M$ be a normal $\Gamma$-covering of a compact complex manifold of complex dimension $n$. Let $g$ be an arbitrary Hermitian metric on $M$ and endow $\widetilde M$ with the pull-back metric $\tilde g$. Then $(D_h^2)^k\to (D_0^2)^k$ in the norm resolvent sense for $h\to0$ and for all $k=0,\dots,n$.
\end{corollary}
\begin{proof}
    This follows immediately from Lemma \ref{lemma covering convergence norm resolvent sense} since $D_h^2=\underset{k=0,\dots,n}{\bigoplus}(D_h^2)^k$.
\end{proof}

\begin{remark}
To prove the following result, we want to apply Theorem \ref{theorem injective projection} to the family $(D_h^2)^k$ for $k=0,\dots,n$. Notice that, in light of Remark \ref{remark more general assumptions}, we can weaken the assumption that $D_h^2\to D_0^2$ in the norm resolvent sense to the convergence in operator norm of $(D_h^2+I)^{-1}\to (D_0^2+I)^{-1}$. This latter convergence can alternatively follow from the convergence in the norm resolvent sense of $D_h\to D_0$, by noticing $(D_h^2+I)^{-1}=(D_h+i)^{-1}(D_h-i)^{-1}$.
\end{remark}

\begin{proposition}\label{proposition l2 frolicher spectral density function}
Let $\pi:\widetilde{M}\to M$ be a normal $\Gamma$-covering of a compact complex manifold of complex dimension $n$. Let $g$ be an arbitrary Hermitian metric on $M$ and endow $\widetilde M$ with the pull-back metric $\tilde g$. Then for all $k=0,\dots,n$ and for every $\tau>0$ there exists an $h_\tau>0$ such that for all $0\le\sigma<\tau$ and $0<h<h_\tau$ 
\[
N_{h,\Gamma}^k(\sigma)\le N_{0,\Gamma}^k(\tau).
\]
\end{proposition}
\begin{proof}
By Lemma \ref{lemma covering convergence norm resolvent sense} and Theorem \ref{theorem injective projection} we know that
\[
E_{0,\tau|\im (E_{h,\sigma}^k)}^k:\im (E_{h,\sigma}^k)\to\im (E_{0,\tau}^k)
\]
is an injective morphism of $\Gamma$-modules. By Proposition \ref{proposition properties von neumann dimension}, this implies the statement.
\end{proof}

We are now ready to prove the $L^2$ Fr\"olicher inequalities.

\begin{theorem}\label{theorem l2 frolicher}
Let $\pi:\widetilde{M}\to M$ be a normal $\Gamma$-covering of a compact complex manifold of complex dimension $n$. Then for all $k=0,\dots,n$
\begin{equation}\label{equation inequality frolicher l2 main}
b^k_{\Gamma}(M)\le \sum_{p+q=k}h^{p,q}_{\delbar,\Gamma}(M).
\end{equation}
\end{theorem}
\begin{proof}
Let us fix an arbitrary Hermitian metric $g$ on $M$ and endow $\widetilde M$ with the pull-back metric $\tilde g$. Now the claim is obtained by applying Propositions \ref{proposition popovici 1} and \ref{proposition l2 frolicher spectral density function} with $\sigma=0$ and taking the limit for $\tau\to0$.
\end{proof}

\begin{corollary}
Let $M$ be a compact complex manifold. Then
\[
b^k(M)\le\sum_{p+q=k}h^{p,q}_\delbar(M).
\]
\end{corollary}
\begin{proof}
Just take $\pi:M\to M$ to be the trivial covering and apply Theorem \ref{theorem l2 frolicher}.
\end{proof}

\begin{corollary}\label{corollary injective projection}
    In the same assumptions of Proposition \ref{proposition l2 frolicher spectral density function}, if $\im (D_0^2)^k$ is closed, then there exists  $h'>0$ such that for all $0<h<h'$ the projection
    \[
    P_{|\ker (D_h^2)^k}:\ker (D_h^2)^k\to\ker(D_0^2)^k
    \]
    is injective.
\end{corollary}
\begin{proof}
    This is Corollary \ref{corollary closed image}.
\end{proof}

Similarly to $L^2 \bar H^k_{dR,\Gamma}(M)$ and $L^2 \bar H^k_{d_h,\Gamma}(M)$, we define the reduced $L^2$ Dolbeault cohomology space as
\[
L^2 \bar H^{p,q}_{\delbar,\Gamma}(M):=\frac{L^2\Lambda^{p,q}\widetilde M\cap\ker\delbar}{\overline{\im\delbar}}.
\]

\begin{corollary}\label{corollary injection l2 cohomology}
    In the same assumptions of Proposition \ref{proposition l2 frolicher spectral density function}, if $\im (D_0^2)^k$ is closed, then there exists an explicit injection of Hilbert $\Gamma$-modules
    \[
    Q_{\Gamma,\tilde g,h}:L^2 \bar H^k_{dR,\Gamma}(M)\to \bigoplus_{p+q=k} L^2 \bar H^{p,q}_{\delbar,\Gamma}(M).
    \]
\end{corollary}
\begin{proof}
The injection $Q_{\Gamma,\tilde g,h}$ is defined as the composition of the following maps: the isomorphism \eqref{equation isomorphism de rham dh} of Proposition \ref{proposition popovici 1}
    \[
    L^2 \bar H^k_{dR,\Gamma}(M)\simeq L^2 \bar H^k_{d_h,\Gamma}(M),
    \]
    the isomorphism given by the $L^2$ Hodge decomposition for the operator $d_h$
    \[
    L^2 \bar H^k_{d_h,\Gamma}(M)\simeq \ker(D_h^2)^k,
    \]
    the injective projection of Corollary \ref{corollary injective projection}
    \[
    \ker(D_h^2)^k\to\ker(D_0^2)^k,
    \]
    and the isomorphism given by the $L^2$ Hodge decomposition for the operator $\delbar$
    \[
    \ker(D_0^2)^k\to\bigoplus_{p+q=k} L^2 \bar H^{p,q}_{\delbar,\Gamma}(M).\qedhere
    \]
\end{proof}

Non-compact and non-K\"ahler examples of $\widetilde M$ having closed image $\im D_0^2$ can be found as the universal covering of the product \cite[Corollary 2.15]{BL} of a compact K\"ahler hyperbolic manifold \cite{G} with a compact non-K\"ahler complex manifold with finite fundamental group.

\begin{corollary}\label{corollary compact injection cohomology}
Let $(M,g)$ be a compact Hermitian manifold. Then for $h>0$ sufficiently small there is an explicit injection
    \[
    Q_{g,h}: H^k_{dR}(M)\to \bigoplus_{p+q=k}  H^{p,q}_{\delbar}(M).
    \]
\end{corollary}
\begin{proof}
Just take $\pi:M\to M$ to be the trivial covering and apply Corollary \ref{corollary injection l2 cohomology}.
\end{proof}

\subsection{Further spectral consequences}

In this subsection we prove Corollary \ref{corollary introduction spectral}.

We start by comparing the Hodge Laplacian associated to a suitably rescaled Hilbert product $\lla,\rra_{\omega_h}$ both to the Hodge Laplacian associated to the standard $L^2$ product and to the twisted Hodge Laplacian $D_h^2$. It follows that the spectral density functions $N_{1,\Gamma}^k$ and $N_{h,\Gamma}^k$ are dilatationally equivalent.
\begin{proposition}\label{proposition popovici 2}
Let $\pi:\widetilde{M} \rightarrow M$ be a normal $\Gamma$-covering of a compact complex manifold $M$ of complex dimension $n$. Let $g$ be an arbitrary Hermitian metric on $M$ and endow $\widetilde M$ with the pull-back metric $\tilde g$. Let us fix $h\in(0,1)$. Let $L^2_{\omega_h}\Lambda^{k}\widetilde M$ be the Hilbert space given by $L^2\Lambda^{k}\widetilde M$ endowed with the product $\lla,\rra_{\omega_h}$ induced by linearity from the following product on $(p,q)$-forms: for all $\alpha,\beta\in L^2\Lambda^{p,q}\widetilde M$
\[
\lla \alpha,\beta\rra_{\omega_h}:=h^{2(p-n)}\lla\alpha,\beta\rra.
\]
Let $d^{*}_{\omega_h}$ be the formal adjoint of $d$ with respect to $\lla,\rra_{\omega_h}$, and $\Delta_{\omega_h}^k:=dd^{*}_{\omega_h}+d^{*}_{\omega_h}d$ be the associated Laplacian on $L^2_{\omega_h}\Lambda^{k}\widetilde M$. Let $\{E_{\omega_h,\lambda}^k\}_{\lambda\in\R}$ be the spectral family of $\Delta_{\omega_h}$ and $N_{\omega_h,\Gamma}^k(\lambda):=\dim_\Gamma(\im(E^k_{\omega_h,\lambda}))$ its spectral density function. Then for all $\lambda\in\R$
\[
\im(E^k_{\omega_h,\lambda})\simeq \im(E^k_{h,\lambda})
\]
are isomorphic as Hilbert $\Gamma$-modules, so that 
\begin{equation*}
N_{\omega_h,\Gamma}^k(\lambda)=N_{h,\Gamma}^k(\lambda).
\end{equation*}
Moreover, there exists $C\in(0,1]$ such that for all $\lambda\in\R$
\begin{equation}\label{equation quasi isometry spectral}
N_{1,\Gamma}^k(C\lambda)\le N_{\omega_h,\Gamma}^k(\lambda)\le N_{1,\Gamma}^k(C^{-1}\lambda).
\end{equation}
In particular, $N_{1,\Gamma}^k$ and $N_{h,\Gamma}^k$ are dilatationally equivalent: there exists $C\in(0,1]$ such that for all $\lambda\in\R$
\begin{equation}\label{equation quasi isometry spectral 2}
N_{1,\Gamma}^k(C\lambda)\le N_{h,\Gamma}^k(\lambda)\le N_{1,\Gamma}^k(C^{-1}\lambda).
\end{equation}
\end{proposition}
\begin{proof}
To prove $\im(E^k_{\omega_h,\lambda})\simeq \im(E^k_{h,\lambda})$, first note that by \cite[(iii) of Lemma 2.7]{Po} we have
\[
(D_h^2)^k=d_hd_h^*+d_h^*d_h=\theta_h \Delta_{\omega_h}^k \theta_h^{-1}.
\]
A quick direct verification may be obtained from the expressions of the formal adjoints \cite[(i),(ii) of Lemma 2.7]{Po}
\begin{align*}
\theta^*_h=\theta_h,&& d_h^*=\theta_h^{-1}d^*\theta_h,&& d^*_{\omega_h}=\theta_h^{-2}d^*\theta_h^2.
\end{align*}
Now, notice that
\[
U_h:=h^{-n}\theta_h: L^2_{\omega_h}\Lambda^k\widetilde M\to L^2\Lambda^k\widetilde M
\]
is an isometric isomorphism of Hilbert $\Gamma$-moduli, and
\[
(D_h^2)^k=U_h \Delta_{\omega_h}^k U_h^{-1}
\]
as self-adjoint operators. From the uniqueness of the spectral family, it follows
\[
E^k_{h,\lambda}=U_h E^k_{\omega_h,\lambda} U_h^{-1}.
\]
In particular 
\[
U_h:\im(E^k_{\omega_h,\lambda})\to \im(E^k_{h,\lambda})
\]
is an isomorphism of Hilbert $\Gamma$-modules. 

To prove \eqref{equation quasi isometry spectral}, we follow the proof of \cite[Proposition 5.1]{ES}. The statement cannot be applied directly since the spectral density function $N_{\omega_h,\Gamma}^k$ is not naturally induced by a $\Gamma$-invariant Riemannian metric on $\widetilde M$. However, the proof is achieved via \cite[Lemma 5.3]{Hi}, which can be applied since $\lla,\rra$ and $\lla,\rra_{\omega_h}$ induce the same topology on $L^2\Lambda^k\widetilde M$, as required in \cite[Section 5.1]{Hi}.
\end{proof}

We recall that a non-negative self-adjoint operator $D$ on a Hilbert space $H$ has \emph{entirely positive spectrum} if $\inf(\spec(D))>0$. In the same setting as before, note that, for any $h\in[0,1]$, there exists $\rho>0$ such that $N_{h,\Gamma}^k(\rho)=0$ iff $(D_h^2)^k$ has entirely positive spectrum.

\begin{lemma}\label{lemma l2 frolicher spectral density function}
Let $\pi:\widetilde{M}\to M$ be a normal $\Gamma$-covering of a compact complex manifold of complex dimension $n$. Let $g$ be an arbitrary Hermitian metric on $M$ and endow $\widetilde M$ with the pull-back metric $\tilde g$. Then for all $k=0,\dots,n$ and $\tau>0$ there exists $C\in(0,1]$ such that for all $0\le\sigma<\tau$
\[
N_{1,\Gamma}^k(C\sigma)\le N_{0,\Gamma}^k(\tau).
\]
\end{lemma}
\begin{proof}
It is enough to combine the inequalities of Proposition \ref{proposition l2 frolicher spectral density function}
\[
N_{h,\Gamma}^k(\sigma)\le N_{0,\Gamma}^k(\tau)
\]
and of \eqref{equation quasi isometry spectral 2} in Proposition \ref{proposition popovici 2}
\[
N_{1,\Gamma}^k(C\sigma)\le N_{h,\Gamma}^k(\sigma).\qedhere
\]
\end{proof}

\begin{theorem}\label{theorem positive spectrum}
Let $\pi:\widetilde{M}\to M$ be a normal $\Gamma$-covering of a compact complex manifold. Let $g$ be an arbitrary Hermitian metric on $M$ and endow $\widetilde M$ with the pull-back metric $\tilde g$. If the Dolbeault Laplacian $(D_0^2)^k$ has entirely positive spectrum, then the same holds for the Hodge Laplacian $(D_1^2)^k$.
\end{theorem}
\begin{proof}
By Lemma \ref{lemma l2 frolicher spectral density function}, if there exists $\tau>0$ such that $N_{0,\Gamma}^k(\tau)=0$, then there exists $0<\sigma<\tau$ such that $N_{1,\Gamma}^k(\sigma)=0$. 
\end{proof}

By further assuming that equality holds in the $L^2$-Fr\"olicher inequality, we are able to prove a similar implication to the one of Theorem \ref{theorem positive spectrum}, replacing the positivity of the spectrum with a spectral gap at 0.

\begin{theorem}\label{theorem spectral gap}
Let $\pi:\widetilde{M}\to M$ be a normal $\Gamma$-covering of a compact complex manifold of complex dimension $n$ and assume that equality holds in \eqref{equation inequality frolicher l2 main}. Let $g$ be an arbitrary Hermitian metric on $M$ and endow $\widetilde M$ with the pull-back metric $\tilde g$. If there is a spectral gap at 0 for $(D_0^2)^k$, then the same holds for $(D_1^2)^k$.
\end{theorem}
\begin{proof}
By Proposition \ref{proposition l2 frolicher spectral density function}, for all $\tau>0$ there exists $h_\tau>0$ such that for all $0<h<h_\tau$ and $0<\sigma<\tau$ we have $N_{h,\Gamma}^k(\sigma)\le N_{0,\Gamma}^k(\tau)$. Assuming the equality in \eqref{equation inequality frolicher l2 main}, and taking into account  $N_{1,\Gamma}^k(0)=N_{h,\Gamma}^k(0)$ by Proposition \ref{proposition popovici 1}, we have
\begin{equation}\label{equation spectral gap}
N_{h,\Gamma}^k(\sigma)-N_{h,\Gamma}^k(0)\le N_{0,\Gamma}^k(\tau)-N_{0,\Gamma}^k(0).
\end{equation}
Note that $(D_0^2)^k$ has a spectral gap at 0 if and only if there exists $\tau>0$ such that the right hand side of \eqref{equation spectral gap} is zero, while $(D_h^2)^k$ has a spectral gap at 0 if and only if there exists $\sigma>0$ such that the left hand side of \eqref{equation spectral gap} is zero. In particular, it follows from \eqref{equation spectral gap} that if $(D_0^2)^k$ has a spectral gap at 0, then $(D_h^2)^k$ has a spectral gap at 0. Using again $N_{1,\Gamma}^k(0)=N_{h,\Gamma}^k(0)$ together with \eqref{equation quasi isometry spectral}, \textit{i.e.}, the fact that the spectral density functions $N_{1,\Gamma}^k$ and $N_{h,\Gamma}^k$ are dilatationally equivalent, we obtain that $(D_h^2)^k$ has a spectral gap at 0 if and only if $(D_1^2)^k$ has a spectral gap at 0. This completes the proof.
\end{proof}

\subsection{$L^2$ Kodaira-Spencer inequality}

We end this section by proving another Fr\"olicher-type inequality of $L^2$ invariants on 2-forms.
Given a compact complex manifold $M$, in \cite[Lemma 4.1]{KM} (cf. \cite[Lemme 3.3]{S}) it is proved
\begin{equation}\label{equation inequality kodaira spencer}
b^2(M)\le 2h^{0,2}_\delbar(M)+h^{1,1}_{BC}(M),
\end{equation}
which is an equality when $M$ is K\"ahler.
Given $M=M_0$, \eqref{equation inequality kodaira spencer} is then used, together with the upper-semi-continuity of $2h^{0,2}_\delbar(M_t)$ and $h^{1,1}_{BC}(M_t)$ along a small deformation of the complex structure $M_t$, to prove that small deformations of compact K\"ahler manifolds are still K\"ahler.

We generalise \eqref{equation inequality kodaira spencer} to the $L^2$ setting. The $L^2$ Bott-Chern number $h^{1,1}_{{BC},\Gamma}(M)$ will be precisely defined in the proof.
\begin{theorem}\label{theorem kodaira spencer l2}
Given a normal $\Gamma$-covering $\pi:\widetilde M\to M$ of a compact complex manifold $M$, we have
\begin{equation*}\label{equation l2 inequality kodaira spencer}
b^2_\Gamma(M)\le 2h^{0,2}_{\delbar,\Gamma}(M)+h^{1,1}_{{BC},\Gamma}(M),
\end{equation*}
which is an equality if $M$ is K\"ahler.
\end{theorem}
\begin{proof}
We follow the same strategy of \cite[Remark 9.9, Proposition 9.10]{HP} but for the exact sequence of \cite[Lemme 3.3]{S} instead of the Varouchas exact sequences. 

First note that by Corollary \ref{corollary shubin essentially self adjoint} and \cite[Theorem 7.6]{HP}, the minimal and maximal closed extensions of $\del\delbar$ and $\delbar^*\del^*$ coincide; therefore we will again abuse notation and denote these unique operators by the same symbols $\del\delbar$ and $\delbar^*\del^*$. Let us set
\begin{align*}
L^2\bar H^{1,1}_\Gamma:=\frac{L^2\Lambda^{1,1}\widetilde M\cap\c{\im d}}{\c{\im\del\delbar}},&& L^2\bar H^{p,q}_{BC,\Gamma}:=\frac{L^2\Lambda^{p,q}\widetilde M\cap\ker d}{\c{\im\del\delbar}}.
\end{align*}
Then the following is an exact sequence
\begin{equation*}\label{equation exact sequence KS}
0\to L^2\bar H^{1,1}_\Gamma\to L^2\bar H^{1,1}_{BC,\Gamma}\to L^2\bar H^2_{dR,\Gamma}\overset{u}{\to} \c{L^2\bar H^{0,2}_{\delbar,\Gamma}}\oplus L^2\bar H^{0,2}_{\delbar,\Gamma}\to \coker u\to 0,
\end{equation*}
where the first three maps are induced by inclusions and the last three maps are induced by projections. There are isomorphisms, induced by the identity, from suitable Hilbert $\Gamma$-modules and the spaces in the above sequence:
\begin{align*}
&L^2\Lambda^{1,1}\widetilde M\cap\c{\im d}\cap\ker\delbar^*\del^*\simeq L^2\bar H^{1,1}_\Gamma, && L^2\Lambda^{1,1}\widetilde M\cap\ker d\cap\ker\delbar^*\del^*\simeq L^2\bar H^{1,1}_{BC,\Gamma},\\
&L^2\Lambda^{2}_\C\widetilde M\cap\ker d\cap\ker d^*\simeq L^2\bar H^{2}_{dR,\Gamma}, && L^2\Lambda^{0,2}\widetilde M\cap\ker \delbar\cap\ker\delbar^*\simeq L^2\bar H^{0,2}_{\delbar,\Gamma}.
\end{align*}
Let us call these Hilbert $\Gamma$-modules respectively $L^2\H^{1,1}_{\Gamma}$, $L^2\H^{1,1}_{BC,\Gamma}$, $L^2\H^{2}_{dR,\Gamma}$ and $L^2\H^{0,2}_{\delbar,\Gamma}$. It follows that the following is a commutative diagram between sequences
\begin{equation*}
\begin{tikzcd}[row sep=small, column sep=small]
0\arrow[r]&L^2\H^{1,1}_{\Gamma}\arrow[r]\arrow[d,"\simeq"]&L^2\H^{1,1}_{BC,\Gamma}\arrow[r,"r"]\arrow[d,"\simeq"]&L^2\H^{2}_{dR,\Gamma}\arrow[r,"s"]\arrow[d,"{\simeq}"]&\c{L^2\H^{0,2}_{\delbar,\Gamma}}\oplus L^2\H^{0,2}_{\delbar,\Gamma}\arrow[r]\arrow[d,"{\simeq}"]&\coker s\arrow[r]\arrow[d,"\simeq"]&0\\
0\arrow[r]&L^2\bar H^{1,1}_\Gamma\arrow[r]&L^2\bar H^{1,1}_{BC,\Gamma}\arrow[r]&L^2\bar H^2_{dR,\Gamma}\arrow[r,"u"]&\c{L^2\bar H^{0,2}_{\delbar,\Gamma}}\oplus L^2\bar H^{0,2}_{\delbar,\Gamma}\arrow[r]&\coker u\arrow[r]&0
\end{tikzcd}
\end{equation*}
where in the top line the first two maps are inclusions, the last two maps are projections, and $r$ and $s$ are defined as follows. Take $\alpha\in L^2\H^{2}_{dR,\Gamma}$ and decompose it by bidegrees as $\alpha=\alpha^{2,0}+\alpha^{1,1}+\alpha^{0,2}$, then $s(\alpha):=P_{L^2\H^{\bullet,\bullet}_{\delbar,\Gamma}}(\alpha^{2,0}+\alpha^{0,2})$, where $P_{L^2\H^{\bullet,\bullet}_{\delbar,\Gamma}}$ denotes the orthogonal projection given by the $L^2$ Hodge decompositions of Section \ref{section preliminaries}. Similarly, the map $r$ is defined as the orthogonal projection on $L^2\H^{2}_{dR,\Gamma}$ restricted to $L^2\H^{1,1}_{BC,\Gamma}$.
We can verify that the top line is a sequence of maps of Hilbert $\Gamma$-modules, which is also exact since the bottom line is exact and the diagram is commutative. Therefore by a generalisation of Proposition \ref{proposition properties von neumann dimension} (cf. \cite[Lemma 9.6]{HP})
\begin{align*}
\dim_\Gamma(L^2\H^{1,1}_{\Gamma})
+\dim_\Gamma(L^2\H^{2}_{dR,\Gamma})
+\dim_\Gamma(\coker u')=\\
=\dim_\Gamma(L^2\H^{1,1}_{BC,\Gamma})
+\dim_\Gamma(\c{L^2\H^{0,2}_{\delbar,\Gamma}}\oplus L^2\H^{0,2}_{\delbar,\Gamma})
\end{align*}
from which the inequality of the statement follows directly. If $M$ is K\"ahler, the equality follows from \cite[Theorem 8.4]{HP}.
\end{proof}

\begin{corollary}
Given a compact complex manifold $M$, we have
\begin{equation*}
b^2(M)\le 2h^{0,2}_{\delbar}(M)+h^{1,1}_{{BC}}(M),
\end{equation*}
which is an equality if $M$ is K\"ahler.
\end{corollary}
\begin{proof}
Just take $\pi:M\to M$ to be the trivial covering and apply Theorem \ref{theorem kodaira spencer l2}.
\end{proof}

\section{Open problems}\label{section open problems}

\subsection{Characterisation of a $L^2$-$\del\delbar$-Lemma}
Let us consider a normal $\Gamma$-covering $\pi:\widetilde M\to M$ of a compact complex manifold $M$.
Combining inequality \eqref{equation l2 inequality hodge abc} with Theorem \ref{theorem l2 frolicher}, we find
\begin{equation}\label{equation intro l2 inequality betti abc} 2b^k_\Gamma(M)\le\sum_{p+q=k}h^{p,q}_{A,\Gamma}(M)+h^{p,q}_{BC,\Gamma}(M).
\end{equation}
Note that, when $M$ is K\"ahler, then by \cite[Theorem 8.4]{HP} we have equality in \eqref{equation intro l2 inequality betti abc}. More precisely, if the maps in the diagram of \cite[Proposition 7.14]{HP} (it is a generalisation of diagram \eqref{equation diagram maps cohomology}) are isomorphisms, then we have equality in \eqref{equation intro l2 inequality betti abc}.
As discussed in the introduction, equality in \eqref{equation angella tomassini inequality betti abc} on a compact complex manifold is equivalent to the $\del\delbar$-Lemma.
\begin{problem}
    Is equality in \eqref{equation intro l2 inequality betti abc} equivalent to any sort of $L^2$-$\del\delbar$-Lemma? \textit{E.g.}, is it equivalent to the fact that the maps in the diagram of \cite[Proposition 7.14]{HP} are isomorphisms?
\end{problem}


\subsection{Upper-semi-continuity of $L^2$ complex invariants along deformations}\label{subsection deformations}
Given a compact complex manifold $M_0$, a family of deformations $M_t$ (we refer to \cite[Chapter 4, Section 4]{KM} for the definitions) and a family of normal $\Gamma$-coverings $\pi_t:\widetilde M_t\to M_t$, the following problem is open.

\begin{problem}
Do the $L^2$ Hodge, Aeppli and Bott-Chern numbers behave upper-semi-continuously (cf. \cite[Chapter 4, Theorem 4.3]{KM})? Namely, is it true that
\[
h^{p,q}_{\delbar,\Gamma_t}(M_t)\le h^{p,q}_{\delbar,\Gamma_0}(M_0)
\]
for $t$ small enough, and similarly for $L^2$ Aeppli and Bott-Chern numbers?
\end{problem}
It seems reasonable that Theorem \ref{theorem injective projection} may find another application in solving this problem.

\end{document}